\documentclass[11pt,letterpaper]{amsart}

\usepackage{amssymb}
\usepackage{enumerate}
\usepackage{bbm}

\newtheorem{theorem}{Theorem}[section]
\newtheorem{proposition}[theorem]{Proposition}
\newtheorem{corollary}[theorem]{Corollary}
\newtheorem{lemma}[theorem]{Lemma}

\theoremstyle{definition}
\newtheorem*{remark*}{Remark}

\newcommand{\FF}{\mathbb{F}}
\newcommand{\CC}{\mathbb{C}}

\newcommand{\NN}{\mathbb{N}}
\newcommand{\QQ}{\mathbb{Q}}

\newcommand{\Fc}{\mathfrak{F}}

\newcommand{\1}{\mathbbm{1}}
\newcommand{\e}{\epsilon}

\newcommand{\floor}[1]{\lfloor#1\rfloor}
\newcommand{\ceil}[1]{\lceil#1\rceil}

\newcommand{\abs}[1]{\lvert#1\rvert}
\newcommand{\bigabs}[1]{\big\lvert#1\big\rvert}

\newcommand{\biggabs}[1]{\bigg\lvert#1\bigg\rvert}
\newcommand{\Biggabs}[1]{\Bigg\lvert#1\Bigg\rvert}

\DeclareMathOperator{\ord}{ord}
\DeclareMathOperator{\Tr}{Tr}
\DeclareMathOperator{\N}{N}
\DeclareMathOperator{\Gal}{Gal}

\begin{document}

\title{Highly nonlinear functions over finite fields}

\author{Kai-Uwe Schmidt}
\address{Department of Mathematics, Paderborn University, Warburger Str.\ 100, 33098 Paderborn, Germany.}
\email{kus@math.upb.de}

\date{16 September 2019}


\begin{abstract}
We consider a generalisation of a conjecture by Patterson and Wiedemann from 1983 on the Hamming distance of a function from $\FF_q^n$ to $\FF_q$ to the set of affine functions from $\FF_q^n$ to $\FF_q$. We prove the conjecture for each $q$ such that the characteristic of $\FF_q$ lies in a subset of the primes with density $1$ and we prove the conjecture for all $q$ by assuming the generalised Riemann hypothesis. Roughly speaking, we show the existence of functions for which the distance to the affine functions is maximised when $n$ tends to infinity. This also determines the asymptotic behaviour of the covering radius of the $[q^n,n+1]$ Reed-Muller code over~$\FF_q$ and so answers a question raised by Leducq in 2013. Our results extend the case $q=2$, which was recently proved by the author and which corresponds to the original conjecture by Patterson and Wiedemann. Our proof combines evaluations of Gauss sums in the semiprimitive case, probabilistic arguments, and methods from discrepancy theory.
\end{abstract}

\maketitle

\thispagestyle{empty}


\section{Introduction and results}

The Hamming distance of two functions $g,h:\FF_q^n\to\FF_q$ is
\[
d(g,h)=\#\{y\in\FF_q^n:g(y)\ne h(y)\}.
\]
We define the \emph{nonlinearity} of $g:\FF_q^n\to\FF_q$ to be
\begin{equation}
N(g)=\min_hd(g,h),   \label{eqn:def_nonlinearity}
\end{equation}
where the minimum is over all $q^{n+1}$ affine functions $h$ from $\FF_q^n$ to $\FF_q$. We are interested in functions with largest nonlinearity. Accordingly define $\rho_q(n)$ to be the maximum of $N(g)$ over all functions $g$ from $\FF_q^n$ to $\FF_q$.
\par
The number $\rho_2(n)$ equals the covering radius of binary Reed-Muller code of order one $R_2(1,n)$~\cite{HelKloMyk1978} and in general $\rho_q(n)$ is the covering radius of the appropriate generalisation $R_q(1,n)$ over $\FF_q$~\cite{Led2013}. The determination of the covering radius of $R_q(1,n)$ appears to be one of the most mysterious problems in coding theory~\cite{Slo1987},~\cite{Led2013}. We refer to~\cite{KasLinPet1968} for~background on Reed-Muller codes over $\FF_q$ and to~\cite{CohHonLitLob1997} for background on the covering radius of codes in general and its combinatorial and geometric significance.
\par
It is convenient to use the normalisation
\[
\mu_q(n)=\frac{q^{n-1}(q-1)-\rho_q(n)}{q^{n/2-1}}.
\]
It is known that
\begin{equation}
1\le\mu_q(n+2k)\le\mu_q(n)   \label{eqn:bounds_mu}
\end{equation}
for all prime powers~$q$ and all positive integers~$n$ and~$k$. This was proved in~\cite{HelKloMyk1978} for $q=2$ and in~\cite[Proposition~11 and Lemma~19]{Led2013} for all~$q$. It is not difficult to see that $\mu_q(2)=1$ and so $\mu_q(n)=1$ for all even $n$, as shown in~\cite[Corollary 1]{HelKloMyk1978} for $q=2$ and \cite[Corollary~13]{Led2013} for all~$q$. 
\par 
We are interested in the case that $n$ is odd. It is readily verified~\cite[p.~1594]{Led2013} that $\mu_q(1)=\sqrt{q}$ and therefore
\begin{equation}
1\le \mu_q(n)\le \sqrt{q}   \label{eqn:elementary_bounds_mu}
\end{equation}
for all prime powers $q$ and all positive integers~$n$. It is known that $\mu_2(n)=\sqrt{2}$ for each $n\in\{3,5,7\}$~\cite{Myk1980}. Patterson and Wiedemann~\cite{PatWie1983} improved the upper bound in~\eqref{eqn:elementary_bounds_mu} for $q=2$ to
\[
\mu_2(n)\le \tfrac{27}{32}\sqrt{2}= 1.19\dots\quad\text{for each $n\ge 15$}
\]
and, more recently, Kavut and Y\"ucel~\cite{KavYuc2010} showed that
\[
\mu_2(n)\le \tfrac{7}{8}\sqrt{2}=1.23\dots\quad\text{for each $n\ge 9$}.
\]
A famous conjecture by Patterson and Wiedemann~\cite{PatWie1983} asserts that
\[
\lim_{n\to\infty}\mu_2(n)=1
\]
and this conjecture was recently proved in~\cite{Sch2019}.
\par
This paper concerns the case that $q>2$. Leducq~\cite{Led2013} herself was able to improve the upper bound in~\eqref{eqn:elementary_bounds_mu} for $q=3$, by showing that $\mu_3(3)=\tfrac{2}{3}\sqrt{3}$ and so
\[
\mu_3(n)\le \tfrac{2}{3}\sqrt{3}=1.15\dots\quad\text{for each $n\ge 3$}.
\]
This suggests that for $q>2$ a similar phenomenon occurs as in the case $q=2$ and indeed we prove a corresponding result for many values of $q$.
\begin{theorem}
\label{thm:main}
Let $q$ be a power of a prime $p$ and suppose that there is another prime $r>3$ such that $r\equiv 3\pmod 4$ and $-p$ is a primitive root modulo $r^2$. Then $\lim_{n\to\infty}\mu_q(n)=1$.
\end{theorem}
\par
We list possible primes $r$ satisfying the conditions of Theorem~\ref{thm:main} for the first 15 primes $p$:
\[
\renewcommand{\arraystretch}{1.5}
\begin{array}{c||ccccccccccccccc}
p & 2 & 3 & 5 & 7 & 11 & 13 & 17 & 19 & 23 & 29 & 31 & 37 & 41 & 43 & 47\\\hline
r & 7 & 23 & 11 & 31 & 7 & 23 & 19 & 31 & 7 & 23 & 11 & 7 & 23 & 19 & 11
\end{array}
\]
For each prime $r$, there are $\phi(\phi(r^2))$ primitive roots modulo $r^2$ and by Dirichlet's theorem on primes in arithmetic progressions, each of the corresponding $\phi(\phi(r^2))$ congruence classes modulo $r^2$ contains a fraction of $1/\phi(r^2)$ of all primes. Hence, by taking a prime $r>3$ with $r\equiv 3\pmod 4$, the condition of Theorem~\ref{thm:main} is satisfied for all $p$ in a subset of the primes with density
\[
\frac{\phi(\phi(r^2))}{\phi(r^2)}=\frac{\phi(r-1)}{r},
\]
where $\phi$ is Euler's totient function. For example, for $2/7$ of all primes $p$, we can take $r=7$ in Theorem~\ref{thm:main}. 
\par
It is known from~\cite{Xie1987} and~\cite{Col1990} that there are infinitely many primes of the form $r=2\ell+1$, where $\ell\ge 3$ is an odd number with at most three prime factors. Let $r_k$ be the $k$-th prime of this form. By the Chinese Remainder Theorem, the density of primes $p$ such that the condition in Theorem~\ref{thm:main} is satisfied for one of the primes $r_1,\dots,r_k$ is $d_k$, where $d_k$ can be recursively defined by $d_1=\phi(r_1-1)/r_1$ and
\[
d_i=d_{i-1}+\frac{\phi(r_i-1)}{r_i}-d_{i-1}\cdot\frac{\phi(r_i-1)}{r_i}
\]
for all $i\ge 2$. Since $r_k-1$ has a bounded number of prime factors, $\phi(r_k-1)/r_k$ is bounded from below by some positive number and hence we have $\lim_{k\to\infty}d_k=1$. We therefore obtain the following corollary of Theorem~\ref{thm:main}.
\begin{corollary}
\label{cor:main}
We have $\lim_{n\to\infty}\mu_q(n)=1$ for all powers $q$ of a prime $p$ lying in a subset of the primes with density $1$. 
\end{corollary}
\par
We shall see that the conclusion of Corollary~\ref{cor:main} can be proved for all prime powers $q$ if one can show that, for each prime $p$, there are infinitely many primes $r\equiv 3\pmod 4$ such that $-p$ is a primitive root modulo $r$. This is known to be true conditionally under the Generalised Riemann Hypothesis (GRH) and gives the following result.
\begin{theorem}
\label{thm:main_GRH}
Assume GRH. Then we have $\lim_{n\to\infty}\mu_q(n)=1$ for all prime powers $q$.
\end{theorem}
\par
For the proof of our results we use a semiprobabilistic construction. We present this construction in the next section (Proposition~\ref{pro:construction}) and then show how our main results follow from this result. The proof that this construction gives the desired properties uses methods from number theory and discrepancy theory and the details are contained in Sections~\ref{sec:Gauss} and~\ref{sec:Discrepancy}. The overall structure of the proof is based on the idea of~\cite{Sch2019} to prove Theorem~\ref{thm:main} for~$q=2$. However, in the general case, several additional ideas are crucially involved.


\section{Proof overview}

For a function $g:\FF_q^n\to\FF_q$, we define the normalisation
\begin{equation}
\mu(g)=\frac{q^{n-1}(q-1)-N(g)}{q^{n/2-1}},   \label{eqn:normalisation}
\end{equation}
where $N(g)$ is the nonlinearity of $g$, given in~\eqref{eqn:def_nonlinearity}. Hence 
\[
\mu_q(n)=\min_g\mu(g),
\]
where the minimum is over all functions $g$ from~$\FF_q^n$ to $\FF_q$. For every $\e>0$, we shall identify functions $f:\FF_q^n\to\FF_q$, which satisfy $\mu(f)\le 1+\e$ when~$n$ is sufficiently large. The construction is semiprobabilistic; it mimics the partial spread construction of so-called bent functions~\cite{Dil1974}, but leaves some freedom, which will bring in probabilistic methods in the proof of our main results.  
\par
Henceforth we identify $\FF_q^n$ with the field $\FF_{q^n}$. Let $H$ be a (multiplicative) subgroup of $\FF_{q^n}^*$ of index $v$. Let $T$ be a union of
\[
q(q-1)\left\lfloor\frac{v}{q(q-1)}\right\rfloor
\]
cosets of $H$ such that, if the coset $aH$ is contained in $T$, then the coset $\lambda aH$ is contained in $T$ for each $\lambda\in\FF_q^*$. Put~$S=\FF_{q^n}\setminus T$. Note that $v$ is not divisible by $q$ and so $S\setminus \{0\}$ is a union of at least $1$ and at most $q^2-q-1$ cosets of~$H$. We consider functions $f:\FF_{q^n}\to\FF_q$ of the form
\begin{equation}
f(y)=\begin{cases}
f_T(y) & \text{for $y\in T$}\\
f_S(y) & \text{for $y\in S$},
\end{cases}   \label{eqn:def_f}
\end{equation}
where $f_T$ is a function from $T$ to $\FF_q$ and $f_S$ is a function from $S$ to $\FF_q$. The function $f_T$ is defined such that $f_T$ takes on every value of $\FF_q$ equally often and such that
\begin{equation}
f_T(\lambda ay)=f_T(y)\quad\text{for each $\lambda\in\FF_q^*$, each $a\in H$, and each $y\in T$}.   \label{eqn:def_F_T}
\end{equation}
That is, $f_T$ is constant on the cosets of $\FF_q^*$ and also constant on the cosets of $H$. The function $f_S$ will be determined later.
\par
Recall that $\ord_m(a)$ for integers $m$ and $a$ with $m>0$ and $\gcd(a,m)=1$ is the smallest positive integer $t$ such that $m\mid a^t-1$. Note that, if we fix~$v$, then for every multiple $n$ of $\ord_v(q)$, there exists a subgroup of $\FF_{q^n}^*$ of index~$v$. In particular, if $p$ is the characteristic of $\FF_q$, then $\ord_v(q)$ divides $\ord_v(p)$, and so such a subgroup exists for every multiple $n$ of $\ord_v(p)$.
\begin{proposition}
\label{pro:construction}
Let $e$ be a positive integer, let $p$ be the characteristic of~$\FF_q$, and suppose that $r>3$ is another prime such that $r\equiv 3\pmod 4$ and~$-p$ is a primitive root modulo $r^e$. Put $v=r^e$. Then there is an odd multiple $n$ of $\ord_v(p)$ and a function $f_S$ such that the function $f$ defined in~\eqref{eqn:def_f} satisfies
\[
\mu(f)\le 1+309\,q^{5/2}\,\sqrt{\frac{\log(2q^2v)}{v}}.
\]
\end{proposition}
\par
\begin{remark*}
With the notation as in Proposition~\ref{pro:construction}, we have that $-1$ is a nonsquare modulo $v$, which implies that
\[
\ord_v(p)=\tfrac{1}{2}\phi(v)=\tfrac{1}{2}(r-1)r^{e-1}.
\]
Hence $\ord_v(p)$ is odd. Therefore $f$ is a function on an extension of $\FF_q$ of odd degree.
\end{remark*}
\par
Before we prove Proposition~\ref{pro:construction} we shall first deduce Theorems~\ref{thm:main} and~\ref{thm:main_GRH} from Proposition~\ref{pro:construction}. Recall from elementary number theory (see~\cite[p.~102]{NivZucMon1991}, for example) that the condition in Theorem~\ref{thm:main} implies that $-p$ is a primitive root modulo $r^e$ for all positive integers~$e$. We can therefore take $e$, and hence $v$, in Proposition~\ref{pro:construction} arbitrarily large. Using~\eqref{eqn:bounds_mu} and~$\mu_q(2)=1$, we then obtain Theorem~\ref{thm:main}.
\par
To deduce Theorem~\ref{thm:main_GRH}, we use the following special case of a result by Moree~\cite{Mor2008}.
\par
\begin{proposition}[{\cite[Theorem~1.3]{Mor2008}}]
\label{pro:primes}
Assume GRH. Let $p$ be a prime. Then the density of primes $r\equiv 3\pmod 4$ such that $-p$ is a primitive root modulo $r$ is 
\[
\frac{A}{2}\left(1-(-1)^{\frac{p-1}{2}}\frac{1}{p^2-p-1}\right)
\]
for odd $p$ and $A/2$ for $p=2$, where
\[
A=\prod_{\text{$r$ prime}}\left(1-\frac{1}{r(r-1)}\right)=0.373955\dots
\]
is Artin's constant
\end{proposition}
\par
Now for fixed $q$, Proposition~\ref{pro:primes} implies, conditional on GRH, the existence of infinitely many primes~$r$ for which we can apply Proposition~\ref{pro:construction} with $e=1$. Using again~\eqref{eqn:bounds_mu} and~$\mu_q(2)=1$, we then obtain Theorem~\ref{thm:main_GRH}.
\par
To prove Proposition~\ref{pro:construction}, we shall turn the problem of estimating the nonlinearity of a function into a problem of estimating certain character sums. Recall that, for a finite field extension $K/F$, the \emph{trace function} $\Tr_{K/F}:K\to F$ is given by
\[
\Tr_{K/F}(y)=\sum_{\sigma\in\Gal(K/F)}\sigma(y)
\]
for each $y\in K$. We define $\eta$ and $\psi$ to be the canonical additive characters of $\FF_q$ and $\FF_{q^n}$, respectively. Denoting by $p$ the characteristic of $\FF_q$, we have
\[
\eta(y)=\exp(2\pi i\Tr_{\FF_q/\FF_p}(y)/p)
\]
for each $y\in\FF_q$ and
\begin{equation}
\psi(y)=\eta(\Tr_{\FF_{q^n}/\FF_q}(y))   \label{eqn:def_psi}
\end{equation}
for each $y\in\FF_{q^n}$.
\par
The \emph{Fourier transform} of a function $g:\FF_{q^n}\to\FF_q$ is defined to be the function $\widehat g:\FF_{q^n}\times \FF_q\to\CC$ given by
\[
\widehat g(a,\lambda)=\frac{1}{q^{n/2}}\sum_{y\in \FF_{q^n}}\eta(\lambda g(y))\overline{\psi(ay)}
\]
for each $a\in\FF_{q^n}$ and each $\lambda\in\FF_q$.
\par
The following lemma gives the relationship between the nonlinearity of a function and its Fourier transform.
\begin{lemma}
\label{lem:nonlinerity_char_sum}
For every function $g:\FF_{q^n}\to \FF_q$ we have 
\begin{equation}
\mu(g)=\max_{a\in\FF_{q^n}}\max_{b\in\FF_q}\;\sum_{\lambda\in\FF_q^*}\overline{\eta(\lambda b)}\;\widehat g(\lambda a,\lambda).   \label{eqn:nonlinearity_char_sum}
\end{equation}
\end{lemma}
\begin{proof}
For every $z\in\FF_q$, we have
\[
\frac{1}{q}\sum_{\lambda\in\FF_q}\eta(\lambda z)=\begin{cases}
1 & \text{for $z=0$}\\
0 & \text{otherwise}.
\end{cases}
\]
Therefore, for every function $h:\FF_{q^n}\to\FF_q$, we have
\begin{align*}
d(g,h)&=q^n-\frac{1}{q}\sum_{y\in\FF_{q^n}}\sum_{\lambda\in\FF_q}\eta(\lambda(g(y)-h(y)))\\
&=q^{n-1}(q-1)-\frac{1}{q}\sum_{\lambda\in\FF_q^*}\sum_{y\in\FF_{q^n}}\eta(\lambda(g(y)-h(y))).
\end{align*}
Now notice that the affine functions from $\FF_{q^n}$ to $\FF_q$ are precisely the $q^{n+1}$ functions $h_{a,b}$ for $a\in\FF_{q^n}$ and $b\in\FF_q$, given by
\[
h_{a,b}(y)=\Tr_{\FF_{q^n}/\FF_q}(ay)+b.
\]
Therefore 
\begin{align*}
d(g,h_{a,b})&=q^{n-1}(q-1)-\frac{1}{q}\sum_{\lambda\in\FF_q^*}\overline{\eta(\lambda b)}\sum_{y\in\FF_{q^n}}\eta(\lambda g(y))\overline{\psi(\lambda ay)}\\
&=q^{n-1}(q-1)-q^{n/2-1}\sum_{\lambda\in\FF_q^*}\overline{\eta(\lambda b)}\;\widehat{g}(\lambda a,\lambda)
\end{align*}
and the lemma follows from the definition~\eqref{eqn:def_nonlinearity} of the nonlinearity of $g$ and the normalisation~\eqref{eqn:normalisation}.
\end{proof}
\par
The strategy for our proof of Proposition~\ref{pro:construction} is to apply Lemma~\ref{lem:nonlinerity_char_sum} to the function $f$ appearing in Proposition~\ref{pro:construction}. We then bound the contributions to $\widehat f(a,\lambda)$ coming from $f_T$ and $f_S$ separately. Accordingly we define
\begin{align*}
\widehat f_T(a,\lambda)&=\frac{1}{q^{n/2}}\sum_{y\in T}\eta(\lambda f_T(y))\overline{\psi(ay)}\\
\widehat f_S(a,\lambda)&=\frac{1}{q^{n/2}}\sum_{y\in S}\eta(\lambda f_S(y))\overline{\psi(ay)},
\end{align*}
so that $\widehat f(a,\lambda)=\widehat f_T(a,\lambda)+\widehat f_S(a,\lambda)$ for all $a\in\FF_{q^n}$ and all $\lambda\in\FF_q$. Proposition~\ref{pro:construction} will then follow in a straightforward way from Lemma~\ref{lem:nonlinerity_char_sum} and the forthcoming Propositions~\ref{pro:bound_f_T} and~\ref{pro:bound_f_S}.


\section{The function $f_T$}
\label{sec:Gauss}

Recall that $H$ is a subgroup of $\FF_{q^n}^*$ of index $v$ and $T$ is a union of cosets of~$\FF_q^*$ and also a union of cosets of $H$. By definition, the function $f_T:T\to\FF_q$ takes on every value of $\FF_q$ equally often and is constant on cosets of $\FF_q^*$ and constant on cosets of $H$, as given in~\eqref{eqn:def_F_T}.
\par
For a multiplicative character $\chi$ of $\FF_{q^n}$, the \emph{Gauss sum} $G(\chi)$ is defined to be
\[
G(\chi)=\sum_{y\in \FF_{q^n}^*}\chi(y)\psi(y),
\]
where as before $\psi$ is the canonical additive character of $\FF_{q^n}$. It is well known that $\abs{G(\chi)}=q^{n/2}$ if $\chi$ is nontrivial (which means that $\chi(y)\ne 1$ for some $y\in\FF_{q^n}^*$)~\cite[Theorem 5.11]{LidNie1997}.
\par
Our starting point for the analysis of $\widehat f_T$ is the following lemma.
\par
\begin{lemma}
\label{lem:gauss_sum_fourier}
Let $\e>0$ and suppose that, for all nontrivial multiplicative characters $\chi$ of $\FF_{q^n}$ of order dividing~$v$, we have
\[
\biggabs{\frac{G(\chi)}{q^{n/2}}+1}\le\e.
\]
Then we have
\[
\sum_{\lambda\in\FF_q^*}\overline{\eta(\lambda b)}\;\widehat f_T(\lambda a,\lambda)\le 1+\e vq
\]
for all $a\in\FF_{q^n}$ and all $b\in\FF_q$.
\end{lemma}
\begin{proof}
Since $f_T$ takes on every value of $\FF_q$ equally often, we have $\widehat f_T(0,\lambda)=0$ for each $\lambda\in\FF_q^*$. Hence we may assume that $a\in\FF_{q^n}^*$. Let~$R$ be a set of representatives of the cosets of $H$ belonging to $T$. For the moment fix $\lambda\in\FF_q^*$. Then we have
\begin{align}
q^{n/2}\,\widehat f_T(\lambda a,\lambda)&=\sum_{y\in T}\eta(\lambda f_T(y))\overline{\psi(\lambda ay)}   \nonumber\\
&=\sum_{z\in R}\sum_{x\in H}\eta(\lambda f_T(z))\overline{\psi(\lambda axz)}   \nonumber\\
&=\sum_{z\in R}\eta(\lambda f_T(z))\sum_{y\in\FF_{q^n}}\1_H(y)\overline{\psi(\lambda ayz)},   \label{eqn:sigma_f}
\end{align}
where $\1_H$ is the indicator of $H$ on $\FF_{q^n}$, so that
\[
\1_H(y)=\begin{cases}
1 & \text{for $y\in H$}\\
0 & \text{otherwise}.
\end{cases}
\]
Let $\chi$ be a multiplicative character of $\FF_{q^n}$ of order $v$. Then
\begin{equation}
\1_H(y)=\frac{1}{v}\sum_{j=0}^{v-1}\chi^j(y)\quad\text{for each $y\in \FF_{q^n}^*$}   \label{eqn:Fourier_indicator}
\end{equation}
and for all $c\in\FF_{q^n}^*$ we have
\begin{align*}
\sum_{y\in \FF_{q^n}}\1_H(y)\overline{\psi(cy)}&=\frac{1}{v}\sum_{j=0}^{v-1}\sum_{y\in\FF_{q^n}^*}\chi^j(y)\overline{\psi(cy)}\\
&=\frac{1}{v}\sum_{j=0}^{v-1}\chi^j(c^{-1})\sum_{y\in\FF_{q^n}^*}\chi^j(y)\overline{\psi(y)}\\
&=\frac{1}{v}\sum_{j=0}^{v-1}\chi^j(c)\sum_{y\in\FF_{q^n}^*}\overline{\chi^j(y)\psi(y)}\\
&=\frac{1}{v}\sum_{j=0}^{v-1}\chi^j(c)\overline{G(\chi^j)}.
\end{align*}
Substitute into~\eqref{eqn:sigma_f} to obtain
\[
\widehat f_T(\lambda a,\lambda)=\frac{1}{q^{n/2}}\;\frac{1}{v}\sum_{z\in R}\eta(\lambda f_T(z))\sum_{j=0}^{v-1}\chi^j(\lambda az)\overline{G(\chi^j)}.
\]
Now write $G(\chi^j)=q^{n/2}(-1+\gamma_j)$, so that $\abs{\gamma_j}\le \e$ for all $j\in\{1,\dots,v-1\}$ by our assumption. Since $\lambda\in\FF_q^*$ and so
\[
\sum_{z\in R}\eta(\lambda f_T(z))=0
\]
by the definition of $f_T$, we obtain
\[
\widehat f_T(\lambda a,\lambda)=M(a,\lambda)+E(a,\lambda),
\]
where
\begin{align*}
M(a,\lambda)&=-\frac{1}{v}\sum_{z\in R}\eta(\lambda f_T(z))\sum_{j=0}^{v-1}\chi^j(\lambda az)\\
E(a,\lambda)&=\frac{1}{v}\sum_{z\in R}\eta(\lambda f_T(z))\sum_{j=0}^{v-1}\chi^j(\lambda az)\,\overline{\gamma_j}.
\end{align*}
From~\eqref{eqn:Fourier_indicator} we find that
\[
M(a)=-\sum_{z\in R}\eta(\lambda f_T(z))\,\1_H(\lambda az).
\]
Since $f_T$ is constant on cosets of $H$ by definition~\eqref{eqn:def_F_T}, we find that
\[
M(a)=\begin{cases}
-\eta(\lambda f_T((\lambda a)^{-1})) & \text{for $(\lambda a)^{-1}\in T$}\\
0                        & \text{otherwise}.
\end{cases}
\]
Since $a^{-1}\in T$ if and only if $(\lambda a)^{-1}\in T$ and since~$f_T$ is constant on cosets of $\FF_q^*$ by definition~\eqref{eqn:def_F_T}, we obtain
\[
M(a)=\begin{cases}
-\eta(\lambda f_T(a^{-1})) & \text{for $a^{-1}\in T$}\\
0                        & \text{otherwise}.
\end{cases}
\]
Hence, for all $b\in\FF_q$, we have
\[
\sum_{\lambda\in\FF_q^*}\overline{\eta(\lambda b)}\,M(a,\lambda)=\begin{cases}
-(q-1) & \text{for $a^{-1}\in T$ and $f_T(a^{-1})=b$}\\
1      & \text{for $a^{-1}\in T$ and $f_T(a^{-1})\ne b$}\\
0      & \text{otherwise}.
\end{cases}
\]
On the other hand, by the triangle inequality we can bound $\abs{E(a,\lambda)}$ by $\e v$ for all $\lambda\in\FF_q^*$ and therefore obtain by the triangle inequality
\[
\Biggabs{\sum_{\lambda\in\FF_q^*}\overline{\eta(kb)}\,E(a,\lambda)}\le \e vq,
\]
as required.
\end{proof}
\par
The following explicit evaluation of certain Gauss sums~\cite[Proposition~4.2]{Lan1997} (see also~\cite[Theorem~4.1]{YanXia2010}) will help us to control the error term in Lemma~\ref{lem:gauss_sum_fourier}.
\par
\begin{lemma}[{\cite[Proposition~4.2]{Lan1997}}]
\label{lem:gauss_sum_index_two}
Let $d$ be a positive integer, let $p$ be a prime, and suppose that $r>3$ is another prime such that $r\equiv 3\pmod 4$ and $-p$ is a primitive root modulo $r^d$. Write $k=\phi(r^d)/2$, let $\tau$ be a multiplicative character of $\FF_{p^k}$ of order~$r^d$, and let $h$ be the class number of $\QQ(\sqrt{-r})$. Then
\[
G(\tau)=\frac{1}{2}(a+b\sqrt{-r})p^{(k-h)/2},
\]
where $a$ and $b$ are integers satisfying $a,b\not\equiv 0\pmod p$, $a^2+b^2r=4p^h$, and $ap^{(k-h)/2}\equiv -2\pmod r$.
\end{lemma}
\par
Recall that for a finite field extension $K/F$, the \emph{norm function} $\N_{K/F}:K\to F$ is defined by
\[
\N_{K/F}(y)=\prod_{\sigma\in\Gal(K/F)}\sigma(y)
\]
for each $y\in K$. Every multiplicative character $\tau$ of $\FF_q$ can be lifted to a multiplicative character $\chi$ of $\FF_{q^s}$ by defining 
\[
\chi(y)=\tau(\N_{\FF_{q^s}/\FF_q}(y))
\]
for each $y\in\FF_{q^s}$. Note that, if $d$ is a divisor of $q-1$, then this lifting is an isomorphism between the character subgroups of order $d$ of $\FF_q^*$ and $\FF_{q^s}^*$.
\par
The well known Davenport-Hasse Theorem gives the relationship between the two Gauss sums $G(\tau)$ and $G(\chi)$.
\begin{lemma}[{\cite[Theorem 5.14]{LidNie1997}}]
\label{lem:DHT}
Let $\tau$ be a multiplicative character of $\FF_q$ and suppose that $\tau$ is lifted to a multiplicative character $\chi$ of $\FF_{q^s}$. Then
\[
G(\chi)=-(-G(\tau))^s.
\]
\end{lemma}
\par
Now we obtain the following lemma as a corollary to Lemma~\ref{lem:gauss_sum_index_two}.
\begin{lemma}
\label{lem:Gauss_sum_evaluations}
Let $e$ and $d$ be integers satisfying $1\le d\le e$ and let $p$ be the characteristic of $\FF_q$. Suppose that $r>3$ is another prime such that $r\equiv 3\pmod 4$ and $-p$ is a primitive root modulo $r^e$. Write $m=\phi(r^e)/2$ and $q=p^t$ and let $h$ be the class number of $\QQ(\sqrt{-r})$. Then there are nonzero integers $a$ and $b$ such that
\[
\frac{G(\chi)}{q^{m/2}}=-\left(-\frac{a\pm b\sqrt{-r}}{2p^{h/2}}\right)^{t\cdot r^{e-d}}
\]
for all multiplicative characters $\chi$ of $\FF_{q^m}$ of order~$r^d$, where the sign can depend on $\chi$.
\end{lemma}
\begin{proof}
Note that $-p$ is also a primitive root modulo $p^d$. Write $k=\phi(q^d)/2$ and let $\tau$ be the multiplicative character of $\FF_{p^k}$ of order $r^d$ such that $\chi$ is the lifted character of $\tau$. Lemma~\ref{lem:gauss_sum_index_two} implies that there are nonzero integers~$a$ and $b$ such that
\[
G(\tau)=\frac{1}{2}(a\pm b\sqrt{-r})p^{(k-h)/2},
\]
where the sign can depend on $\chi$. By Lemma~\ref{lem:DHT} we have
\[
\frac{G(\chi)}{q^{m/2}}=-\left(-\frac{a\pm b\sqrt{-r}}{2p^{h/2}}\right)^{tm/k}
\]
and the lemma follows since $m/k=\phi(r^e)/\phi(r^d)=r^{e-d}$.
\end{proof}
\par
The next lemma gives the desired control for the error term in Lemma~\ref{lem:gauss_sum_fourier}.
\begin{lemma}
\label{lem:Gauss_sum_error}
Let $e$ be a positive integer and let $p$ be the characteristic of~$\FF_q$. Suppose that $r>3$ is another prime such that $r\equiv 3\pmod 4$ and $-p$ is a primitive root modulo $r^e$. Write $m=\phi(r^e)/2$ and let $\e>0$. Then there is an infinite set $I$ of odd positive integers such that, for all $s\in I$ and all nontrivial multiplicative characters~$\chi$ of $\FF_{q^{sm}}$ of order dividing $r^e$, we have
\[
\abs{\arg (-G(\chi))}\le \e.
\]
Here, $\arg(\xi)\in(-\pi,\pi]$ is the principal angle of a nonzero complex number~$\xi$.
\end{lemma}
\begin{proof}
Let $\tau$ be a multiplicative character of $\FF_{q^m}$ of order $r^e$. Since $r>3$, the units in the ring of algebraic integers of $\QQ(\sqrt{-r})$ are $\pm 1$, so that $\pm 1$ are the only roots of unity in $\QQ(\sqrt{-r})$. It then follows from Lemma~\ref{lem:Gauss_sum_evaluations} that $G(\tau)/q^{m/2}$ is not a root of unity. Therefore Weyl's uniform distribution theorem~\cite[Satz~2]{Wey1916} implies that $([G(\tau)/q^{m/2}]^{2i})_{i\in\NN}$, and therefore also $(G(\tau)/q^{m/2}]^{2i+1})_{i\in\NN},$ is uniformly distributed on the complex unit circle. Hence there is an infinite set $I$ of odd positive integers such that
\[
\abs{\arg(-G(\tau)^s)}\le \frac{\e}{r^{e-1}}
\]
for all $s\in I$.
\par
Let $s\in I$ and lift $\tau$ to a multiplicative character $\tau'$ to $\FF_{q^{sm}}$. Then $\tau'$ has order~$r^e$ and Lemma~\ref{lem:DHT} implies $G(\tau')=G(\tau)^s$, so that
\[
\abs{\arg (-G(\tau'))}\le \frac{\e}{r^{e-1}}.
\]
Now let $\chi$ be a multiplicative character of $\FF_{q^{sm}}$ of order $r^d$, where $1\le d\le e$. Then by Lemma~\ref{lem:Gauss_sum_evaluations} we have
\[
\abs{\arg (-G(\chi))}\le r^{e-d}\,\abs{\arg (-G(\tau'))},
\]
which completes the proof.
\end{proof}
\par
We are now in a position to deduce the following result, which controls~$\widehat f_T$ and gives our first desired ingredient for the proof of Proposition~\ref{pro:construction}.
\begin{proposition}
\label{pro:bound_f_T}
Let $e$ be a positive integer and let $p$ be the characteristic of $\FF_q$. Suppose that $r>3$ is another prime such that $r\equiv 3\pmod 4$ and~$-p$ is a primitive root modulo $r^e$. Put $v=r^e$ and let $\e>0$. Then there are infinitely many odd multiples $n$ of $\ord_v(p)$ such that the function $f_T$ satisfies
\[
\sum_{\lambda\in\FF_q^*}\overline{\eta(\lambda b)}\;\widehat f_T(\lambda a,\lambda)\le 1+\e vq
\]
for all $a\in\FF_{q^n}$ and all $b\in\FF_q$.
\end{proposition}
\begin{proof}
Write $m=\phi(v)/2$ and note that $m=\ord_v(p)$. Letting $\e>0$, Lemma~\ref{lem:Gauss_sum_error} implies that there is an infinite set $I$ of odd positive integers such that
\[
\biggabs{\frac{G(\chi)}{q^{sm/2}}+1}\le\e
\]
for all $s\in I$ and all nontrivial multiplicative characters $\chi$ of $\FF_{q^{sm}}$ of order dividing $v$. The desired result then follows from Lemma~\ref{lem:gauss_sum_fourier}.
\end{proof}
\par
We remark that in Proposition~\ref{pro:bound_f_T} the conclusion holds for infinitely many~$n$, which is stronger than what is needed to prove Proposition~\ref{pro:construction}.


\section{The function $f_S$}
\label{sec:Discrepancy}

This section concerns the existence of an appropriate function $f_S:S\to\FF_q$. We shall use the following result that might be also of independent interest in discrepancy theory.
\begin{theorem}
\label{thm:set_partition}
Let $K\ge 2$ be an integer and let $\Fc$ be a family of $M$ subsets of a finite set $X$ with $\abs{X}=N$ and $M\ge N$. Then, for all sufficiently large~$N$, there exists a partition $\{Z_1,Z_2,\dots,Z_K\}$~of~$X$~such~that
\begin{equation}
\max_{Y\in\Fc}\biggabs{\abs{Y\cap Z_i}-\frac{\abs{Y}}{K}}\le 308\sqrt{\frac{N}{K}\log\frac{2KM}{N}}   \label{eqn:subsets_bound}
\end{equation}
for each $i\in\{1,2,\dots,K\}$.
\end{theorem}
\par
The constant in Theorem~\ref{thm:set_partition} can certainly be improved by a more careful analysis. We note that Doerr and Srivastav~\cite[Theorem~3.15]{DoeSri2003} proved a result similar to Theorem~\ref{thm:set_partition}. However, compared to the proof of~\cite[Theorem~3.15]{DoeSri2003}, our proof of Theorem~\ref{thm:set_partition} is completely different and considerably simpler, although both proofs are based on Lemma~\ref{lem:lin_forms} below.
\par
Before we prove Theorem~\ref{thm:set_partition}, we deduce the following result for the existence of an appropriate function~$f_S$, which gives our second desired ingredient for the proof of Proposition~\ref{pro:construction}. Recall that $S$ is a subset of $\FF_{q^n}$ such that $S\setminus\{0\}$ contains at least~$1$ and at most $q^2-q-1$ cosets of a subgroup of~$\FF_{q^n}^*$ of index $v$. Therefore
\begin{equation}
q^n/v\le\abs{S}\le q^{n+2}/v.   \label{eqn:size_S}
\end{equation}
\begin{proposition}
\label{pro:bound_f_S}
For fixed $v$ and all sufficiently large $n$, there is a function $f_S:S\to\FF_q$ such that
\[
\bigabs{\widehat f_S(a,\lambda)}\le 308\,q^{5/2}\,\sqrt{\frac{\log(2q^2v)}{v}}
\]
for all $a\in\FF_{q^n}$ and all $\lambda\in\FF_q^*$.
\end{proposition}
\begin{proof}
For each $a\in\FF_{q^n}$ and each $z\in\FF_q$, define
\begin{equation}
Y_{a,z}=\{y\in S:\Tr_{\FF_{q^n}/\FF_q}(ay)=z\}.   \label{eqn:def_Yaz}
\end{equation}
From Theorem~\ref{thm:set_partition} we find that, for all sufficiently large $\abs{S}$, there exists a partition $\{Z_1,Z_2,\dots,Z_q\}$ of $S$ such that
\begin{equation}
\biggabs{\abs{Y_{a,z}\cap Z_k}-\frac{\abs{Y_{a,z}}}{q}}\le 308\sqrt{\frac{\abs{S}}{q}\log\frac{2q^{n+2}}{\abs{S}}}   \label{eqn:estimate_Y_ij_Z_k}
\end{equation}
for all $a,z,k$. Henceforth suppose that $\abs{S}$ is large enough so that this last estimate holds. For $\FF_q=\{z_1,z_2,\dots,z_q\}$, define $f_S:S\to\FF_q$ by $f_S(y)=z_k$ for $y\in Z_k$. Let $\eta$ be the canonical additive character of $\FF_q$ and let $\lambda\in\FF_q^*$. From~\eqref{eqn:estimate_Y_ij_Z_k} we find that
\[
\Biggabs{\sum_{y\in Y_{a,z}}\eta(\lambda f_S(y))-\frac{\abs{Y_{a,z}}}{q}\sum_{c\in\FF_q}\eta(\lambda c)}\le 308\,\sqrt{\abs{S}q\log(2q^{n+2}/\abs{S})} 
\]
for all $a,z$. Since $\sum_{c\in\FF_q}\eta(\lambda c)=0$, we obtain
\begin{equation}
\Biggabs{\sum_{y\in Y_{a,z}}\eta(\lambda f_S(y))}\le 308\,\sqrt{\abs{S}q\log(2q^{n+2}/\abs{S})}   \label{eqn:estimate_eta_fS} 
\end{equation}
for all $a,z$. We have
\begin{align*}
\widehat f_S(a,\lambda)&=\frac{1}{q^{n/2}}\sum_{y\in S}\eta(\lambda f_S(y))\overline{\psi(ay)}\\
&=\frac{1}{q^{n/2}}\sum_{z\in\FF_q}\overline{\eta(z)}\sum_{y\in Y_{a,z}}\eta(\lambda f_S(y)),
\end{align*}
using~\eqref{eqn:def_psi} and~\eqref{eqn:def_Yaz}. Therefore by the triangle inequality and~\eqref{eqn:estimate_eta_fS} we obtain
\[
\bigabs{\widehat f_S(a,\lambda)}\le \frac{308}{q^{n/2}}\,\sqrt{\abs{S}q^3\log(2q^{n+2}/\abs{S})},
\]
and using~\eqref{eqn:size_S}, we can obtain the required estimate.
\end{proof}
\par
In the remainder of this section we prove Theorem~\ref{thm:set_partition}. We need a classical result from discrepancy theory due to
Spencer~\cite{Spe1985}, which we quote in the following specialised form.
\begin{lemma}[{\cite[Theorem 7]{Spe1985}}]
\label{lem:lin_forms}
Let $\Fc$ be a family of $M$ subsets of a finite set~$X$ with $\abs{X}=N$ and $M\ge N$ and let $\delta$ be a real number. Then, for all sufficiently large $N$, there exists $h:X\to\{-\delta,\delta\}$ such that
\[
\max_{Y\in\Fc}\Biggabs{\sum_{y\in Y}h(y)}\le11\,\delta\sqrt{N\log(2M/N)}.
\]
\end{lemma}
\par
We shall deduce the following result from Lemma~\ref{lem:lin_forms} using an idea of Beck~\cite{Bec1991}.
\begin{lemma}
\label{lem:single_subset}
Let $\Fc$ be a family of $M$ subsets of a finite set $X$ with $\abs{X}=N$ and $M\ge N$ and let $\theta\in[0,1]$. Then, for all sufficiently large $N$, there exists a subset $Z$ of $X$ such that
\[
\max_{Y\in\Fc}\bigabs{\abs{Y\cap Z}-\theta\,\abs{Y}}\le 23\sqrt{N\log(2M/N)}.
\]
\end{lemma}
\begin{proof}
We may assume that $\theta\in[0,\tfrac{1}{2}]$; otherwise we replace~$Z$ by its complement in $X$. The case $\theta=0$ is trivial since we can take $Z$ to be the empty set.
\par
Now assume first that $\theta=\tfrac{1}{2}$. Let $h:X\to\{-1,1\}$ be a function identified in Lemma~\ref{lem:lin_forms} for $\delta=1$. Put 
\[
Z=\{y\in X:h(y)=1\}.
\]
Then by Lemma~\ref{lem:lin_forms} we have, for all sufficiently large $N$,
\[
\max_{Y\in\Fc}\bigabs{\abs{Y\cap Z}-(\abs{Y}-\abs{Y\cap Z})}\le 11\sqrt{N\log(2M/N)},
\]
and so
\[
\max_{Y\in\Fc}\bigabs{\abs{Y\cap Z}-\tfrac{1}{2}\abs{Y}}\le \frac{11}{2}\sqrt{N\log(2M/N)},
\]
as required.
\par
Henceforth assume that $\theta\in(0,\tfrac{1}{2})$. Let $\alpha$ be a real number such that
\[
\cos\alpha=\frac{\theta}{\theta-1}
\]
and let $\Delta$ be the triangle with vertices
\[
e^{2\pi i\alpha},\;1,\;e^{-2\pi i\alpha}.
\]
The triangle $\Delta$ can be decomposed into four triangles that are congruent to~$2^{-1}\Delta$. By iterating this decomposition, we have the chain of partitions
\[
\Delta=\bigcup_{i=1}^4\Delta(1,i)=\bigcup_{i=1}^{4^2}\Delta(2,i)=\cdots=\bigcup_{i=1}^{4^k}\Delta(k,i)=\cdots,
\]
where, for each $i\in\{1,2,\dots,4^k\}$, the triangle $\Delta(k,i)$ is congruent to $2^{-k}\Delta$. Let $t$ be a natural number to be determined later. Then we have
\[
0\in \Delta(t,i_t)\subset \Delta(t-1,i_{t-1})\subset\cdots\subset\Delta(1,i_1)\subset\Delta
\]
for some sequence $i_1,i_2,\dots,i_t$. It will be convenient to write $\Delta=\Delta(0,1)$ and $i_0=1$. 
\par
We now construct functions $h_0,h_1,\dots,h_t:X\to\CC$ such that $h_k(y)$ is a vertex of $\Delta(k,i_k)$ for each $y\in X$. For each $y\in X$, let $h_t(y)$ be a vertex of the small triangle $\Delta(t,i_t)$ with minimum absolute value. Since the diameter of $\Delta$ ist at most $2$, the diameter of $\Delta(t,i_t)$ ist at most $2^{-t+1}$, and so we have
\[
\abs{h_t(y)}\le 2^{-t}
\]
for each $y\in X$. Therefore
\begin{equation}
\max_{Y\in\Fc}\Biggabs{\sum_{y\in Y}h_t(y)}\le N2^{-t}.   \label{eqn:rounding_small_triangle}
\end{equation}
Now let $k\in\{1,2,\dots,t\}$ and suppose that $h_k(y)$ is a vertex of $\Delta(k,i_k)$ for each $y\in X$. Then, for each $y\in X$, the point $h_k(y)$ is either a vertex of $\Delta(k-1,i_{k-1})$ or is a midpoint between two vertices of $\Delta(k-1,i_{k-1})$. We set $h_{k-1}(y)=h_k(y)$ for all $y\in X$, except for those $y\in X$ corresponding to the latter case. The remaining values of $h_{k-1}(y)$ are rounded to one of the neighbouring vertices of $\Delta(k-1,i_{k-1})$ using Lemma~\ref{lem:lin_forms}. Since the diameter of $\Delta(k-1,i_{k-1})$ is at most $2^{-k+2}$, we have for all sufficiently large $N$,
\[
\max_{Y\in\Fc}\Biggabs{\sum_{y\in Y}h_k(y)-\sum_{y\in Y}h_{k-1}(y)}\le 22\cdot2^{-k}\sqrt{N\log(2M/N)}.
\]
Hence by the triangle inequality we have, for all sufficiently large $N$,
\begin{align}
\max_{Y\in\Fc}\Biggabs{\sum_{y\in Y}h_t(y)-\sum_{y\in Y}h_0(y)}&\le \sum_{k=1}^t
\max_{Y\in\Fc}\Biggabs{\sum_{y\in Y}h_k(y)-\sum_{y\in Y}h_{k-1}(y)}   \nonumber\\
&\le \sum_{k=1}^t22\cdot2^{-k}\sqrt{N\log(2M/N)}   \nonumber\\
&\le 22\sqrt{N\log(2M/N)}.   \label{eqn:rounding_medium_triangles}
\end{align}
Applying the triangle inequality once more, we obtain from~\eqref{eqn:rounding_small_triangle}, for all sufficiently large $N$,
\begin{align}
\max_{Y\in\Fc}\Biggabs{\sum_{y\in Y}h_0(y)}&\le 22\sqrt{N\log(2M/N)}+N2^{-t}   \nonumber\\
&\le 23\sqrt{N\log(2M/N)},   \label{eqn:bound_u0}
\end{align}
by choosing $t$ large enough. Now $h_0(y)$ is a vertex of $\Delta$ for each $y\in X$. Put
\[
Z=\{y\in X:h_0(y)=1\}.
\]
Let $Y\in\Fc$ be fixed and assume that $N$ is large enough, so that~\eqref{eqn:bound_u0} holds. By considering the real part of the summation on the left hand side of~\eqref{eqn:bound_u0}, we obtain
\[
\bigabs{(\abs{Y}-\abs{Y\cap Z})\cos\alpha+\abs{Y\cap Z}}\le 23\sqrt{N\log(2M/N)}.
\]
Equivalently we have
\[
\biggabs{\abs{Y\cap Z}-\frac{-\cos\alpha}{1-\cos\alpha}\;\abs{Y}}\le 23\;\frac{\sqrt{N\log(2M/N)}}{1-\cos\alpha}.
\]
Since $\cos\alpha<0$ and
\[
\frac{-\cos\alpha}{1-\cos\alpha}=\theta,
\]
we conclude that $Z$ has the required property.
\end{proof}
\par
It remains to prove Theorem~\ref{thm:set_partition}.
\begin{proof}[Proof of Theorem~\ref{thm:set_partition}]
It will be useful to work with the family of subsets $\Fc'=\Fc\cup\{X\}$ of $X$, so that $\abs{\Fc'}\le M+1$.
\par
First apply Lemma~\ref{lem:single_subset} with $\theta=(1/K)\lfloor K/2\rfloor$ to infer the existence of a subset $A$ of~$X$ such that~$A$ intersects each $Y\in\Fc'$ in roughly~$\theta\abs{Y}$ elements. Then the complement $B$ of~$A$ intersects each $Y\in\Fc'$ in roughly $(1-\theta)\abs{Y}$ elements. The problem is now reduced because it remains to partition $A$ into $\floor{K/2}$ subsets and $B$ into $\ceil{K/2}$ subsets. If necessary, we apply Lemma~\ref{lem:single_subset} to the families of subsets $\Fc'$ restricted to $A$ and $B$ and then proceed iteratively, so that in each step Lemma~\ref{lem:single_subset} is applied with some $\theta\in[1/3,1/2]$, until we obtain a partition $\{Z_1,Z_2,\dots,Z_K\}$ of~$X$ such that each $Z_i$ intersects each $Y\in\Fc'$ in roughly $\abs{Y}/K$ elements.
\par
We now give a quantitative analysis. For every $Z\in\{Z_1,Z_2,\dots,Z_K\}$, there are subsets $W_0,W_1,\dots,W_s$  (with $K\le 2^s<2K$) of $X$ satisfying
\[
Z=W_s\subset W_{s-1}\subset \cdots\subset W_1\subset W_0=X
\]
and numbers $\mu_1,\dots,\mu_s\in[1/3,2/3]$ satisfying $\mu_1\cdots\mu_s=1/K$ such that
\begin{equation}
\bigabs{\abs{W_i\cap Y}-\mu_i\,\abs{W_{i-1}\cap Y}}\le 23\sqrt{\abs{W_{i-1}}\log\frac{2(M+1)}{\abs{W_{i-1}}}}   \label{eqn:estimate_W_i}
\end{equation}
for each $i\in\{1,2,\dots,s\}$, each $Y\in\Fc'$, and all sufficiently large $N$. By the triangle inequality we have
\begin{equation}
\bigabs{\abs{W_j\cap Y}-\mu_1\cdots \mu_j\abs{Y}}\le\sum_{i=1}^j\mu_{i+1}\cdots\mu_j\,\bigabs{\abs{W_i\cap Y}-\mu_i\,\abs{W_{i-1}\cap Y}}   \label{eqn:W_Y_triangle_inequality}
\end{equation}
for each $j\in\{1,2,\dots,s\}$ and each $Y\in\Fc'$. In particular, by taking $Y=X$ we obtain from~\eqref{eqn:estimate_W_i} and~\eqref{eqn:W_Y_triangle_inequality} that
\[
\frac{N}{K}\le\abs{W_j}\le \frac{N}{K}\,\frac{2}{\mu_{j+1}\cdots\mu_s}
\]
for each $j\in\{1,2,\dots,s-1\}$ and all sufficiently large $N$ (with room to spare). Since $\abs{W_0}=N$, these estimates also hold for $j=0$, and so substitution into~\eqref{eqn:estimate_W_i} gives
\[
\bigabs{\abs{W_i\cap Y}-\mu_i\,\abs{W_{i-1}\cap Y}}\le 23\sqrt{\frac{2}{\mu_i\cdots\mu_s}\,\frac{N}{K}\log\frac{2K(M+1)}{N}}
\]
for each $i\in\{1,2,\dots,s\}$, each $Y\in\Fc'$, and all sufficiently large $N$. From~\eqref{eqn:W_Y_triangle_inequality} with $j=s$ we then find that
\begin{align*}
\biggabs{\abs{Z\cap Y}-\frac{\abs{Y}}{K}}&\le23\sum_{i=1}^s\sqrt{\frac{2\mu_{i+1}\cdots\mu_s}{\mu_i}}\sqrt{\frac{N}{K}\log\frac{2K(M+1)}{N}}\\
&\le23\sqrt{\frac{6N}{K}\log\frac{2K(M+1)}{N}}\sum_{i\ge 0}(2/3)^{i/2}
\end{align*}
for all $Y\in\Fc'$ and all sufficiently large $N$, where we have used that $1/3\le \mu_i\le 2/3$ for all $i$. The series equals $\sqrt{3}/(\sqrt{3}-\sqrt{2})$, from which the claimed bound can be obtained.
\end{proof}


\end{document}